\documentclass[a4paper,12pt]{amsart}

	\usepackage{amsmath,amsthm,amsfonts,amssymb,mathrsfs}
   \usepackage{graphicx}
   \usepackage{color}
   \usepackage[nice]{nicefrac}
	\usepackage{dsfont}
	\usepackage{delarray}
	\usepackage{enumerate}

   \newtheorem{theo}{Theorem}[section]
	\newtheorem{prop}[theo]{Proposition}
   \newtheorem{lemm}[theo]{Lemma}
   
   \newtheorem{defi}[theo]{Definition}

	\theoremstyle{remark}
	\newtheorem{rema}[theo]{Remark}

   \usepackage{lineno}
   \usepackage{verbatim}
   \usepackage{fancyhdr}

	\def \K{\mathcal K}

   \def \div{ \nabla \cdot}
   
   \def \Z{\mathbb Z}
   \def \N{\mathbb N}
	
   \newcommand{\dx}{\,\text dx}
   \newcommand{\dy}[1]{\,\text d#1}

	\newcommand\R{\mathbb{R}}

	\def\norm#1#2{\|#1\|_{L^#2(\Omega)}}

\numberwithin{equation}{section}
\numberwithin{theo}{section}

\begin{document}

\title[Stability of solutions to aggregation equation]{Stability of solutions to aggregation equation in bounded domains}

\author{Rafa\l\ Celi\'nski}
\address{
 Instytut Matematyczny, Uniwersytet Wroc\l awski,
 pl. Grunwaldzki 2/4, 50-384 Wroc\-\l aw, POLAND}
\email{Rafal.Celinski@math.uni.wroc.pl}

\date{\today}

\thanks{
Author was supported by the International Ph.D.
Projects Programme of Foundation for Polish Science operated within the
Innovative Economy Operational Programme 2007-2013 funded by European
Regional Development Fund (Ph.D. Programme: Mathematical
Methods in Natural Sciences) and by the MNiSzW grant No. N N201 418839. This is a part of the author Phd dissertation written under supervision of Grzegorz Karch.
}

\begin{abstract} 
We consider the aggregation equation $u_t= \div(\nabla u-u\nabla \mathcal{K}(u))$ in a bounded domain $\Omega\subset \mathbb{R}^d$ with supplemented the Neumann boundary condition and with a nonnegative, integrable initial datum. Here, $\mathcal{K}=\mathcal{K}(u)$ is an integral operator. We study the local and global existence of solutions and we derive conditions which lead us to either the stability or instability of constant solutions. 
\end{abstract}

\bigskip


\maketitle

\section{Introduction}
We consider the initial value problem for the following non-local transport equation
\begin{align}
u_t= \div(\nabla u-u\nabla \K(u)) \quad \text{for}\ x \in \Omega\subset \R^d,\ t > 0, \label{eq}
\end{align}
with supplemented the Neumann boundary conditions {\it i.e}
\begin{align}\label{bound_cond}
\frac{\partial u}{\partial n}=0 \quad \text{for}\ x \in\partial\Omega,\ t>0
\end{align}
and a nonnegative initial datum
\begin{align}\label{eq-ini}
u(x,0)=u_0(x).
\end{align}
Here, the operator $\K(u)=\K(u)(x,t)$ depends linearly on $u$ via the following integral formula
\begin{align}\label{int_op}
\K(u)(x,t)=\int_\Omega K(x,y)u(y,t)\dy{y}
\end{align}
for a certain function $K=K(x,y)$ which we call as an {\it aggregation kernel}. 

There is large number of works considering  the {\it inviscid}  aggregation equation
\begin{align}\label{inviscid}
u_t+\div(u(\nabla K\ast u))=0
\end{align}
in the whole space $\R^d$ which has been used to describe aggregation phenomena in the modelling of animal collective behaviour as well as in some problems
in mechanics of continuous media, for instance, \cite{CMV06,LT04,ME99}. The unknown function $u=u(x,t)\geq 0$ represents either the population density of a species or, in the case of materials applications, a particle density. Equation \eqref{inviscid} was derived from the system of ODE called ``individual cell-based model''  \cite{BV05,S00} representing behaviour of a collection of self-interacting particles via pairwise potential which is describe by aggregation kernel $K$. More precisely, equation \eqref{inviscid} is a continuum limit for a system of particles $X_k(t)$ placed at the point $k$ in time $t$ and evolving by the system of differential equations:
\begin{align*}
\frac{dX_k(t)}{dt}=-\sum_{i\in\Z\backslash\{k\}}\nabla K(X_k(t)-X_i(t)),\quad k\in\Z
\end{align*}
where $K$ is the potential. 

Questions on the global-in-time well-posedness, finite and infinite time blowups, asymptotic behaviour of solutions to equation \eqref{inviscid}, as well as to the equation with an additional diffusion term, have been extensively studied by a number of authors; see {\it e.g.} \cite{BCL09,BL07,BLR11,CDFLS11,KS11,L07,LR09} and reference therein.

One introduces the diffusion term in \eqref{inviscid} to make the model more realistic and to describe the interesting biological (and mathematical as well) phenomenon: competition between aggregation and diffusion, see {\it e.g.} \cite{B95,CPZ04,KS10,N00}. 

In this work, however, our main motivation to study such models is that, in particular case, equation \eqref{eq} corresponds to the parabolic-elliptic system describing chemotaxis, namely: 
\begin{equation}
u_t = \div(\nabla u - u \nabla v) ,  \quad  -\Delta v +av= u, \qquad x\in\Omega,\ t>0 \label{chemo}
\end{equation}
for a positive constant $a$. In this system, the function $u=u(x,t)$
represents the cell density and $v=v(x,t)$ is the concentration of the chemical
attractant which induces a drift force. Here, the function $K(x,y)$ is the Green function of the operator $-\partial_x^2+aI$ on $\Omega$ with the Neumann boundary conditions. Moreover, it is called the Bessel potential and it is singular at the origin if $d\geq 2$. On the other hand, in the one-dimensional case,  when $\Omega=[0,1]$ and $a=1$ this fundamental solution is given by the explicit formula {\it i.e.}
\begin{align}\label{green}
K(x,y)=\frac{1}{2}e^{-|x-y|}+\frac{e^{x+y}+e^{2-x-y}+e^{x-y}+e^{y-x}}{2(e^2-1)}.
\end{align}

In this work we derive some properties of solutions of aggregation equation in a bounded domain under no flux boundary condition \eqref{bound_cond}. The main goal, is to study stability of constant solution. In particular, we derive conditions under which constant solutions to problem \eqref{eq}-\eqref{eq-ini} are either stable or unstable. Here, let us point out that instability result does not depend on dimension of the domain, and cover the case when the aggregation kernel comes from chemotaxis model \eqref{chemo}. Hence, even though solutions are global-in-time and bounded, a constant steady state can be unstable. This mean that even in one-dimensional chemotaxis we can observe the competition between aggregation and diffusion mentioned above. 

For the completeness of exposition we also discuss existence of solutions to \eqref{eq}--\eqref{eq-ini}. In order to do that, we use techniques which are rather standard and well known. In particular, we show that under some general condition on aggregation kernel we can always construct local-in-time solution to \eqref{eq}--\eqref{eq-ini}.  However, some additional regularity assumption on the initial datum have to be imposed if $\nabla_xK$ is in some sense too singular. Moreover, for mildly singular kernels (see Definition \ref{def} for precise statement), problem \eqref{eq}-\eqref{eq-ini} has a global-in-time solution for any nonnegative and integrable initial condition.

\subsection*{Notation}
In this work, the usual norm of the Lebesgue space $L^p (\Omega)$ with respect to the spatial variable is denoted by $\|\cdot\|_p$ for any $p \in [1,\infty]$ and $W^{k,p}(\Omega)$ is the corresponding Sobolev space. 
The letter $C$ corresponds to a generic constants (always independent of $x$ and $t$) which may vary from line to line. Sometimes, we write, {\it e.g.}  $C=C(\alpha,\beta,\gamma, ...)$ when we want to emphasise the dependence of $C$ on parameters~$\alpha,\beta,\gamma, ...$.

\section{Main results and comments.}

\subsection{Stability and instability of constant solutions.}

In this paper, we assume the following conditions on the aggregation kernel
\begin{align}
&\frac{\partial K}{\partial n}(\cdot,y)=0 \quad\text{on}\ \partial\Omega \quad\text{for all}\quad y\in\Omega,\label{v_1}
\end{align}
\begin{align}\label{ass_stab}
\nabla_x\int_\Omega K(x,y)\dy{y}=0,
\end{align}
\begin{multline}
\|\nabla_xK\|_{\infty,q^\prime}\equiv {\rm ess}\sup_{x\in\Omega}\|\nabla_x K(x,\cdot)\|_{q^\prime}+{\rm ess}\sup_{y\in\Omega}\|\nabla_x K(\cdot,y)\|_{q^\prime} <\infty  \label{v_2}\\
 {\rm for\ some }\ q^\prime\in[1,\infty]. 
\end{multline}

\begin{rema}\label{th-mass}
Notice that under the assumptions \eqref{v_1}, a solution of problem \eqref{eq}--\eqref{eq-ini} conserves the integral (the ``mass'') {\it i.e.}
\begin{equation}\label{mass}
\|u(t)\|_1=\int_\Omega u(x,t)\dx = \int_\Omega u_0(x)\dx= \|u_0\|_1\quad \mbox{for all} \quad t\geq 0.
\end{equation}
Indeed, it is sufficient to integrate the equation \eqref{eq} with respect to $x$ and use identities \eqref{bound_cond} and \eqref{v_1}.  Moreover, this solution remains nonnegative if the initial condition is so, due to the maximum principle.
\end{rema}

\begin{rema}\label{const}
Note, that assumption \eqref{ass_stab} implies that every constant function $u\equiv M$ satisfies equation \eqref{eq}. In fact, the chemotaxis model \eqref{chemo} is our main motivation to state this assumption. Indeed, if $(U,V)$ is a stationary solution to \eqref{chemo} then $U$ is constant if and only if $V$ is constant as well. It means that, if the kernel $K$ is the Green function of the operator $-\Delta+aI$ then the term $\nabla\K(U)$ in equation \eqref{eq} for $U=M$, has to be equal $0$ and so, $K$ satisfies \eqref{ass_stab}.
\end{rema}

The main goal of this work is to study stability of constant solution to problem \eqref{eq}-\eqref{eq-ini}. More precisely, we look for sufficient conditions either on the stability of constant solutions or their instability. Our result can be summarise in the following way
\begin{itemize}
\item If the constant solution $u(x,t)=M\geq 0$ of problem \eqref{eq}--\eqref{eq-ini} is sufficiently small, then it is asymptotically stable solution in the linear and nonlinear sense, see Proposition \ref{stab1} and Theorem \ref{stab2} below.
\item If the constant solution $u(x,t)=M\geq 0$ is sufficiently large, then there is a large class of aggregation kernels (which include the kernel coming from chemotaxis system \eqref{chemo}), such that $u(x,t)=M$ is a linearly unstable solution of \eqref{eq}--\eqref{eq-ini}.
\end{itemize}

Thus, we focus on a solution to problem \eqref{eq}-\eqref{eq-ini} in the form
\begin{align*}
u(x,t)=M+\varphi(x,t),
\end{align*}
where $M$ is an arbitrary constant and $\varphi$ is a perturbation. Moreover, we assume that $\int_\Omega \varphi(x,t)\dx=0$ for all $t\geq 0$, to have
\begin{align*}
\int_\Omega u(x,t)\dx=\int_\Omega u_0(x)\dx=\int_\Omega M\dx=M|\Omega|\quad \text{for all}\ t>0.
\end{align*}
Hence, from equation \eqref{eq}, using assumption \eqref{ass_stab}, we obtain the following initial boundary value problem for the perturbation $\varphi$
\begin{align}
&\varphi_t=\Delta\varphi-\div\Big{(}M\nabla\K(\varphi)+ \varphi\nabla\K(\varphi) \Big{)}\label{pre.lin}\\
&\frac{\partial \varphi}{\partial n}=0 \quad \text{for}\ x \in\partial\Omega,\ t>0 \label{pre.lin-bd}\\
&\varphi(x,0)=\varphi_0(x). \label{pre.lin-ini}
\end{align}

We also introduce its linearized counterpart, namely, we skip the term $\div(\varphi\nabla\K(\varphi))$ on the right hand side of \eqref{pre.lin} to obtain
\begin{align}
&\varphi_t=\Delta\varphi-\div\Big{(}M\nabla\K(\varphi)\Big{)} \label{lin}\\
&\frac{\partial \varphi}{\partial n}=0 \quad \text{for}\ x \in\partial\Omega,\ t>0 \label{lin-bd}\\
&\varphi(x,0)=\varphi_0(x). \label{lin-ini}
\end{align}
In the following, we use the linear operator $\mathcal{L}\varphi=-\Delta\varphi+\div\Big{(}M\nabla\K(\varphi)\Big{)}$ with the Neumann boundary conditions, defined via its associated bilinear form
\begin{align}\label{bilin_form}
J(\varphi,\psi)=\int_\Omega\nabla\varphi\cdot\nabla\psi\dx-M\int_\Omega\nabla\K(\varphi)\nabla\psi\dx
\end{align}
for all $\varphi,\psi\in W^{1,2}(\Omega)$.

Here, we recall that a constant $M$ is called a linearily asymptotically stable stationary solution to nonlinear problem \eqref{eq}-\eqref{eq-ini} if the zero solution is an asymptotically stable solution of the linearized problem \eqref{lin}-\eqref{lin-ini}. Moreover, a constant $M$ is called linearily unstable stationary solution to nonlinear problem \eqref{eq}-\eqref{eq-ini} if zero is an unstable solution to linearized problem \eqref{lin}-\eqref{lin-ini}.

\begin{prop}[Linear stability of constant solutions]\label{stab1}
Assume, that the aggregation function $K(x,y)$ satisfy conditions \eqref{v_1} and \eqref{ass_stab}. If, moreover, the operator $\nabla\K:L^2(\Omega)\to L^2(\Omega)$ given by the form $\nabla\K(\varphi)=\int_\Omega\nabla_x K(x,y)\varphi(y)\dy{y}$ is bounded and if 
\begin{align}\label{linear_ass}
M\|\nabla\K\|_{L^2\to L^2}<\sqrt{\lambda_1},
\end{align}
where $\lambda_1$ is the first non-zero eigenvalue of $-\Delta$ on $\Omega$ under the Neumann boundary condition then $M$ is a linearily asymptotically stable stationary solution to problem \eqref{eq}-\eqref{eq-ini}.
\end{prop}

We prove this proposition in Section \ref{stability}. Here, we only emphasise that proof allow us to show the nonlinear stability of constant steady states. Under slightly stronger assumptions imposed on the kernel $K$.

\begin{theo}[Nonlinear stability of constant solution]\label{stab2}
Let the assumptions of Proposition \ref{stab1} hold true. If moreover $\|\nabla_x K\|_{\infty,2}<\infty$ then there exists a positive constant $\eta=\eta(\nabla_xK,M,\Omega)$ such that for every $\varphi_0\in L^2(\Omega)$ satisfying $\|\varphi_0\|_2<\eta$ and $\int_\Omega\varphi_0(x)\dx=0$, the perturbed problem \eqref{pre.lin}-\eqref{pre.lin-ini} has a solution $\varphi\in C([0,\infty),L^2(\Omega))$ such that $\int_\Omega \varphi(x,t)\dx=0$ for all $t>0$. Moreover, we have
\begin{align*}
\|\varphi(t)\|_2\rightarrow 0 \quad\text{as}\quad t\to\infty.
\end{align*}
\end{theo}

Next, we discuss instability of constant solutions.

\begin{theo}[Instability of constant solutions]\label{instab}
Let $w_1=w_1(x)$ be the eigenfunction of $-\Delta$ on $\Omega$ under the Neumann boundary condition corresponding to the first nonzero eigenvalue $\lambda_1$, and such that $\|w_1\|_2=1$. Assume that $\|\nabla\K\|_{L^2\to L^2}<\infty$. If moreover, the aggregation function $K(x,y)$ satisfy 
\begin{align}\label{ass-instab}
\int_\Omega\int_\Omega K(x,y)w_1(y)w_1(x)\dx\dy{y}=A>0,
\end{align}
then for $M>1/A$ the constant solution $M$ of problem \eqref{eq}-\eqref{eq-ini} is linearily unstable stationary solution.
\end{theo}

\begin{rema}
Let us notice that the aggregation function $K$ which comes from chemotaxis model \eqref{chemo} satisfies the condition \eqref{ass-instab}. Indeed, in this case, $K(x,y)$ is a fundamental solution of the operator $-\Delta+aI$ in a bounded domain supplemented with the Neumann boundary conditions. Thus, the function
\begin{align*}
w(x)=\int_\Omega K(x,y)w_1(y)\dy{y}
\end{align*}
satisfies the following equation
\begin{align}\label{rem}
-\Delta w+aw=w_1.
\end{align}
After multiplying equation \eqref{rem} by $w_1$ and integrating over $\Omega$ and using the Neumann boundary condition we obtain
\begin{align*}
-\int_\Omega\Delta ww_1\dx+a\int_\Omega ww_1\dx=\int_\Omega (w_1)^2\dx.
\end{align*}
Obviously, by the definition of $A$, we have $\int_\Omega ww_1\dx=A$. Thus, after integrating by parts we obtain
\begin{align}\label{rem1}
-\int_\Omega w\Delta w_1\dx=1-aA.
\end{align}
Finally, we use the fact that $w_1$ is the eigenfunction of $-\Delta$ to get
\begin{align*}
\lambda_1\int_\Omega ww_1\dx=1-aA,
\end{align*}
which implies that $A=\frac{1}{a+\lambda_1}>0$.
\end{rema}

\begin{rema}
Our stability results on constant steady states corresponds to the well-known results on the global existence versus blow-up of solutions to Keller-Segel system \eqref{chemo}. In particular, for general kernels (see Definition \ref{def} below), where solutions of problem \eqref{eq}--\eqref{eq-ini} are global-in-time, we can still observe the competition between the diffusion and the aggregation.
\end{rema}

\begin{rema}
Let us mention, that the {\it inviscid} aggregation equation \eqref{inviscid} in the whole space $\R^d$ can be formally considered as a gradient flow of the energy functional
\begin{align*}
E(u)=\frac{1}{2}\int_{\R^d}\int_{\R^d} K(x-y)u(x)u(y)\dx\dy{y}
\end{align*}
with respect to the Euclidean Wasserstein distance as introduced in \cite{O01} and generalized to a large class of PDEs in \cite{CMV06} and in \cite{CDFLS11}. We have proved that, in some sense, if this energy functional on the first eigenfunction of $-\Delta$ is positive then sufficiently large constant solutions of the system (1.1)-(1.3) are unstable.
\end{rema}

The proofs of Theorems \ref{stab2} and \ref{instab} are given in Section \ref{stability}.

\subsection {Existence of solutions.}

For the completeness of exposition we also study the existence of solution to \eqref{eq}--\eqref{eq-ini}.
First, let us introduce terminology analogous to that one in \cite{KS11}.

\begin{defi}\label{def}
The aggregation kernel $K:\Omega\times\Omega\rightarrow\R$ is called
\begin{itemize}
\item {\it mildly singular} if $\|\nabla_xK\|_{\infty,q^\prime}<\infty$ for some $q^\prime\in(d,\infty]$;
\item {\it strongly singular} if $\|\nabla_xK\|_{\infty,q^\prime}<\infty$ for some $q^\prime\in[1,d]$ and $\|\nabla_xK\|_{\infty,q^\prime}=\infty$ for every $q^\prime>d$.
\end{itemize}
\end{defi}

Notice that aggregation kernel taken from one dimensional chemotaxis model \eqref{chemo} is mildly singular in the sense stated above. 

We begin our study of  properties of solutions to the initial value problem \eqref{eq}--\eqref{eq-ini} by showing the existence of solutions which depends on the quantity $\|\nabla_xK\|_{\infty,q^\prime}$ defined in \eqref{v_2}.

First, we show that for mildly singular kernels, solutions to the problem \eqref{eq}-\eqref{eq-ini} are global in time.

\begin{theo}[Global existence for mildly singular kernels]\label{local-ex1}
Assume that there exists $q^\prime\in(d,\infty]$ such that $\|\nabla_xK\|_{\infty,q^\prime}<\infty$ where $\|\nabla_xK\|_{\infty,q^\prime}$ is defined in \eqref{v_2}. Denote $q=\frac{q^\prime}{q^\prime-1}\in[1,d/(d-1))$. Then for every initial condition $u_0 \in L^1 (\Omega)$ such that $u_0(x)\geq 0$ and for every $T>0$ problem \eqref{eq}-\eqref{eq-ini} has a unique mild solution in the space 
\begin{align*}
\mathcal{Y}_T = C([0, T], L^1 (\Omega)) \cap \{u: C\big{(}[0,T], L^q(\Omega)\big{)},\ \sup_{0\le t\le T} t^{\frac{d}{2}(1-\frac{1}{q})}\|u\|_q<\infty\}
\end{align*} 
equipped with the norm $\|u\|_{\mathcal{Y}_T}\equiv\sup_{0\leq t\leq T}\|u\|_1+ \sup_{0\leq t\leq T}t^{\frac{d}{2}(1-\frac{1}{q})}\|u\|_q$.
\end{theo}

Next, we show the local-in-time existence of solutions to \eqref{eq}-\eqref{eq-ini} for the case of strongly singular kernels.

\begin{theo}[Local existence for strongly singular kernels]\label{local-ex2}
Assume that there exists $q^\prime\in[1,d]$ such that $\|\nabla_xK\|_{\infty,q^\prime}<\infty$.
Let $q\in[d/(d-1),\infty]$ satisfy $1/q+1/q^\prime=1$. Then for every positive $u_0 \in L^1(\Omega)\cap L^q(\Omega)$ there exists $T=T(\|u_0\|_1,\|u_0\|_q,\|\nabla_xK\|_{\infty,q^\prime})>0$ and a unique mild solution of problem \eqref{eq}--\eqref{eq-ini} in the space 
\begin{align*}
\mathcal{X}_T = C([0, T], L^1 (\Omega)) \cap C([0,T], L^q(\Omega))
\end{align*} 
equipped with the norm $\|u\|_{\mathcal{X}_T}\equiv\sup_{0\leq t\leq T}\|u\|_1+ \sup_{0\leq t\leq T}\|u\|_q$.
\end{theo}

\begin{rema}
Let us mention, that our previous, stability results imply the global-in time existence of solutions for strongly singular kernels provided initial data are sufficiently small. More results on the global-in-time solutions to \eqref{eq}-\eqref{eq-ini} in the whole space $\Omega=\R^d$ can be found {\it e.g.} in \cite{KS11}.
\end{rema}

\begin{rema}
Karch and Suzuki in their work \cite{KS11} studied the {\it viscous} aggregation equation, namely the equation \eqref{eq} considered in the whole space $\R^d$. They show that there are strongly singular kernels (in the sense similar to Definition \ref{def}), such that some solutions blow up in finite time. Moreover, there is a large number of works studying the blow-up of solution to chemotaxis model \eqref{chemo}, see {\it e.g.} \cite{B98,N00b,N00,NSS00} and reference therein as well as the review paper by Horstmann \cite{horst} for additional references.
\end{rema}


\section{Stability and instability of constant solutions}\label{stability}


In our reasoning, we use the following Poincar\'e inequality
\begin{align}\label{poincare}
\lambda_1\int_\Omega\psi^2\dx\leq\int_\Omega|\nabla\psi|^2\dx,
\end{align}
which is valid for all $\psi\in W^{1,2}(\Omega)$ satisfying $\int_\Omega\psi\dx=0$, where $\lambda_1$ is the first non-zero eigenvalue of $-\Delta$ on $\Omega$ under the Neumann boundary condition.


Now, we are in the position to prove the Theorem \ref{stab1}.

\begin{proof}[Proof of Proposition \ref{stab1}]

After multiplying equation \eqref{lin} by $\varphi$ and integrating over $\Omega$ we get
\begin{align*}
\frac{1}{2}\frac{d}{\dy{t}}\|\varphi(\cdot,t)\|^2_2=-\int_\Omega|\nabla\varphi|^2\dx+ M\int_\Omega\nabla\K(\varphi)\nabla\varphi\dx.
\end{align*}
Now, using the Cauchy inequality we obtain
\begin{equation}\label{pr1-e2}
\begin{split}
\frac{1}{2}\frac{d}{\dy{t}}\|\varphi(\cdot,t)\|^2_2&\leq
-\frac{1}{2}\int_\Omega|\nabla\varphi|^2\dx+\frac{M^2}{2}\int_\Omega(\nabla\K(\varphi))^2\dx\\ 
&\leq -\frac{1}{2}\int_\Omega|\nabla\varphi|^2\dx+ \frac{M^2}{2}\|\nabla\K\|_{L^2\to L^2}^2\int_\Omega\varphi^2\dx.
\end{split}
\end{equation}
Finally, we apply the Poincar\'e inequality \eqref{poincare} to get the following differential inequality
\begin{align*}
\frac{d}{\dy{t}}\|\varphi(\cdot,t)\|^2_2\leq \Big{(}-\lambda_1+ M^2\|\nabla\K\|_{L^2\to L^2}^2\Big{)}\|\varphi\|_2^2
\end{align*}
which, under assumption \eqref{linear_ass}, directly leads us to the exponential decay of $\|\varphi(t)\|_2$ as $t\to\infty$.
\end{proof}

\begin{proof}[Proof of Theorem \ref{stab2}]
After multiplying equation \eqref{pre.lin} by $\varphi$ and integrating over $\Omega$ we get
\begin{align}\label{pr2-eq1}
\frac{1}{2}\frac{d}{\dy{t}}\|\varphi\|_2^2=-J(\varphi,\varphi)+ \int_\Omega \varphi\nabla\K(\varphi)\nabla\varphi\dx
\end{align}
where $J$ is the bilinear form defined in \eqref{bilin_form}.
In \eqref{pr1-e2}, we have already got the inequality
\begin{align}\label{pr2-eq2}
-J(\varphi,\varphi)\leq -\frac{1}{2}\int_\Omega|\nabla\varphi|^2\dx+ \frac{M^2}{2}\|\nabla\K\|_{L^2\to L^2}^2\int_\Omega\varphi^2\dx.
\end{align}

To estimate the second (nonlinear) term on the right-hand side of \eqref{pr2-eq1}, we use the \\$\varepsilon$-Cauchy inequality, as follows
\begin{equation}\label{pr2-eq3}
\begin{split}
\int_\Omega\varphi\nabla\K(\varphi)\nabla\varphi\dx&\leq \varepsilon\int_\Omega(\nabla\varphi)^2\dx+ \frac{1}{4\varepsilon}\int_\Omega\varphi^2(\nabla\K(\varphi))^2\dx\\
&\leq\varepsilon\int_\Omega(\nabla\varphi)^2\dx+ \frac{1}{4\varepsilon}\|\nabla\K(\varphi)\|_\infty^2\int_\Omega\varphi^2\dx\\
&\leq\varepsilon\int_\Omega(\nabla\varphi)^2\dx+ \frac{\|\nabla_xK\|_{\infty,2}^2}{4\varepsilon}\Big{(}\int_\Omega\varphi^2\dx\Big{)}^2,
\end{split}
\end{equation}
since
\begin{align*} 
\|\int_\Omega\nabla_x K(\cdot,y)\varphi(y)\dy{y}\|_\infty\leq {\rm ess}\sup_{x\in\Omega} \|\nabla_xK(x,\cdot)\|_2\|\varphi\|_2= \|\nabla_xK\|_{\infty,2}\|\varphi\|_2.
\end{align*}
Applying inequalities \eqref{pr2-eq2} and \eqref{pr2-eq3} in \eqref{pr2-eq1} we obtain
\begin{align*}
\frac{d}{\dy{t}}\int_\Omega\varphi^2\dx\leq &\Big{(}-1+2\varepsilon\Big{)}\int_\Omega(\nabla\varphi)^2\dx\\
&+(M^2\|\nabla\K\|^2_{L^2\to L^2}) \int_\Omega\varphi^2\dx+ \frac{\|\nabla_xK\|_{\infty,2}^2}{2\varepsilon}\Big{(}\int_\Omega\varphi^2\dx\Big{)}^2,
\end{align*}
and finally using Poincar\'e inequality \eqref{poincare} we get the following differential inequality
\begin{align*}
\frac{d}{\dy{t}}\|\varphi\|_2^2\leq \Big{(}\lambda_1(2\varepsilon-1)+M^2\|\nabla\K\|_{L^2\to L^2}^2 \Big{)}\|\varphi\|_2^2+\frac{\|\nabla_xK\|_{\infty,2}^2}{2\varepsilon}\|\varphi\|_2^4.
\end{align*}
Notice, that under assumption \eqref{linear_ass}, we can find $\varepsilon>0$ small enough that the term $\Big{(}\lambda_1(2\varepsilon-1)+ M^2\|\nabla\K\|_{L^2\to L^2}^2 \Big{)}$ is negative. Thus, the proof is complete because every nonnegative solution of the differential inequality $f^\prime\leq -C_1f+C_2f^2$ with $f(t)=\|\varphi(t)\|_2^2$ and with positive constants $C_1$, $C_2$ decays exponentially to zero, provided $f(0)$ is sufficiently small.
\end{proof}

To study the instability of constant solutions, first, we consider eigenvalues of the operator $\mathcal{L}$ defined via its bilinear form \eqref{bilin_form}.

\begin{lemm}\label{rayley}
Let the operator
\begin{align}\label{op_L}
\mathcal{L}\varphi=-\Delta\varphi+\div\Big{(}M\nabla\K(\varphi)\Big{)}
\end{align}
supplemented with the Neumann boundary condition be defined by the associated bilinear form $J(\varphi,\psi)$ given in \eqref{bilin_form} on $W^{1,2}(\Omega)$.
Assume that $\nabla_xK\in L^2(\Omega\times\Omega)$ satisfies \eqref{v_1}. Then, the number
\begin{align}\label{R}
\lambda=\inf_{\substack{\varphi\in W^{1,2}(\Omega)\\ \int_\Omega\varphi\dx=0}}\frac{J(\varphi,\varphi)}{\|\varphi\|_2^2},
\end{align}
is finite and there exists $\tilde{\varphi}\in W^{1,2}(\Omega)$ such that
\begin{align*}
\lambda=
\frac{J(\tilde{\varphi},\tilde{\varphi})}{\|\tilde{\varphi}\|_2^2}.
\end{align*}
Moreover, $\mathcal{L}\tilde{\varphi}=\lambda\tilde{\varphi}$ in the weak sense.
\end{lemm}
\begin{proof}
As usual, in \eqref{R} we may restrict ourselves to the case $\|\varphi\|_2=1$.
Now, let 
\begin{align*}
\mathcal{A}=\{\varphi\in W^{1,2}(\Omega): \|\varphi\|_2=1,  \int_\Omega\varphi\dx=0\}.
\end{align*}

{\it Step 1.}
First we show that $J(\varphi,\varphi)$ is bounded from below on $\mathcal{A}$. Repeating the estimates from the proof of Proposition \ref{stab1} we obtain
\begin{align*}
\Big{|}M\int_\Omega\nabla\K(\varphi)\nabla\varphi\dx\Big{|}\leq \frac{1}{2}\|\nabla\varphi\|_2^2+ \frac{M^2}{2}\|\nabla\K\|^2_{L^2\to L^2}\|\varphi\|_2^2.
\end{align*}
Hence, for every $\varphi\in\mathcal{A}$ we have
\begin{align*}
J(\varphi,\varphi)\geq\frac{1}{2}\|\nabla\varphi\|_2^2- \frac{M^2}{2}\|\nabla\K\|^2_{L^2\to L^2}\|\varphi\|_2^2\geq -\frac{M^2}{2}\|\nabla\K\|^2_{L^2\to L^2}.
\end{align*}

{\it Step 2.}
Let $\{\varphi_\})_{n\in\N}\subset \mathcal{A}$ be a minimizing sequence that is
\begin{align*}
\lambda=\lim_{n\to\infty}J(\varphi_n,\varphi_n).
\end{align*}
We show that $\varphi_n$ is bounded in $W^{1,2}(\Omega)$. Since $\varphi_n$ is the minimizing sequence, there exists a constant $C$ such that
\begin{align*}
C\geq J(\varphi_n,\varphi_n)\geq\frac{1}{2}\|\nabla\varphi_n\|_2^2- \frac{M^2}{2}\|\nabla\K\|^2_{L^2\to L^2},
\end{align*}
so we obtain
\begin{align*}
\|\nabla\varphi_n\|_2^2\leq 2C+M^2\|\nabla\K\|^2_{L^2\to L^2}.
\end{align*}
Thus, using the Rellich compactness theorem we have a subsequence, again denoted by $\varphi_n$, converging to $\tilde{\varphi}$ strongly in $L^2(\Omega)$. Moreover, by the Banach-Alaoglu theorem, we obtain, again up to subsequence, also weak convergence of $\varphi_n$ towards to $\tilde\varphi$ in $W^{1,2}(\Omega)$. 

Notice, that $\tilde\varphi\in\mathcal{A}$. Indeed, by the weak convergence in $W^{1,2}(\Omega)$ we have that $\tilde{\varphi}\in W^{1,2}(\Omega)$ and by the strong convergence in $L^2(\Omega)$ the limit function satisfy $\|\tilde{\varphi}\|_2=1$ and $\int_\Omega\tilde{\varphi}\dx=0$.

{\it Step 3.}
Now, we show that $\lim_{n\to\infty}J(\varphi_n,\varphi_n)= J(\tilde\varphi,\tilde\varphi)$.

First, notice that by the weak convergence of $\nabla\varphi_n$ in $W^{1,2}(\Omega)$ we have
\begin{align}\label{lemm-1}
\liminf_{n\to\infty}\|\nabla\varphi_n\|_2\geq\|\nabla\tilde\varphi\|_2.
\end{align}
Next, by the strong convergence of $\varphi_n$ in $L^2(\Omega)$ and the fact that $\nabla\K:L^2(\Omega)\to L^2(\Omega)$ is linear and bounded it is easy to verify that
\begin{align*}
\nabla\K(\varphi_n)\rightarrow\nabla\K(\tilde\varphi)\quad \text{as}\quad n\to\infty\quad \text{strongly in } L^2(\Omega).
\end{align*}
This property and again the weak convergence of $\tilde\varphi_n$ implies that
\begin{align*}
\int_\Omega\nabla\K(\varphi_n)\nabla\varphi_n\dx\rightarrow
\int_\Omega\nabla\K(\tilde\varphi)\nabla\tilde\varphi\dx\quad\text{as}\quad n\to\infty
\end{align*}
which by estimate \eqref{lemm-1} together with previous step completes the proof of Step 3.

{\it Step 4.}
Finally, we show that the limit function $\tilde\varphi$ satisfies the following eigenvalue problem $\mathcal{L}\tilde\varphi=\lambda\tilde{\varphi}$ in the weak sense, namely
\begin{align*}
J(\tilde\varphi,v)=\lambda\int_\Omega\tilde\varphi v\dx\qquad \text{for all}\ v\in W^{1,2}(\Omega).
\end{align*}
Let us denote
\begin{align*}
f(t)=\frac{J(\tilde\varphi+\varepsilon v,\tilde\varphi+\varepsilon v)} {\int_\Omega(\tilde\varphi+\varepsilon v)^2\dx}
\end{align*}
for any $v\in W^{1,2}$ and $\varepsilon\in\R$. This function is differentiable with respect to $\varepsilon$ near $\varepsilon=0$ and has a minimum at $0$. Hence the derivative vanishes at $\varepsilon=0$, and we get
\begin{align*}
0=f^\prime(0)&=\frac{J(\tilde\varphi,v)}{\int_\Omega(\tilde\varphi)^2\dx}- \frac{J(\tilde\varphi,\tilde\varphi)}{\int_\Omega(\tilde\varphi)^2\dx} \frac{\int_\Omega\tilde\varphi v\dx}{\int_\Omega(\tilde\varphi)^2\dx}
=J(\tilde\varphi,v)-\lambda\int_\Omega\tilde\varphi v\dx.
\end{align*}
Hence the proof of Lemma \ref{rayley} is finished.
\end{proof}

Now we are in the position to prove the Theorem \ref{instab}.
\begin{proof}[Proof of Theorem \ref{instab}]
As a standard practise, we show that under our assumptions, the linear operator
$\mathcal{L}$ defined by the form \eqref{op_L} has a negative eigenvalue $\lambda$. Then, the function
$\varphi(x,t)=e^{-\lambda t}\tilde{\varphi}(x)$
with the eigenfunction $\tilde{\varphi}$ of $\mathcal{L}$ corresponding to the eigenvalue $\lambda$, is a solution of the linearized problem \eqref{lin}-\eqref{lin-ini} such that
\begin{align*}
\|\varphi(\cdot,t)\|_2=e^{-\lambda t}\|\tilde{\varphi}\|_2\xrightarrow{t\to\infty}\infty.
\end{align*}

To do so, we use the definition of an eigenvalue of operator $\mathcal{L}$ from Lemma \ref{rayley}. In view of \eqref{R}, to prove that $\lambda<0$, it suffices to show that there exist $\varphi\in W^{1,2}(\Omega)$ that 
\begin{align*}
J(\varphi,\varphi)<0.
\end{align*}
Here, we choose $\varphi(x)=w_1(x)$, where $w_1$ is the eigenfunction of $-\Delta$ on $\Omega$ under the Neumann boundary condition satisfying $\int_\Omega w_1^2\dx=1$ and corresponding to the first non-zero eigenvalue $\lambda_1$. Then, we obtain the following relation
\begin{align*}
J(w_1,w_1)&=\int_\Omega(\nabla w_1(x))^2\dx- M\int_\Omega\int_\Omega \nabla_x K(x,y)w_1(y)\nabla w_1(x)\dy{y}\dx\\
&=\lambda_1\int_\Omega (w_1(x))^2\dx-M\lambda_1\int_\Omega\int_\Omega K(x,y)w_1(y)w_1(x)\dy{y}\dx.
\end{align*}
Now, since $\lambda_1>0$ and $\int_\Omega (w_1)^2\dx=1$, using assumption \eqref{ass-instab} and choosing $M>1/A$ we complete the proof.

\end{proof}


\section{Existence of solutions}\label{sec-ex}


We  construct  local-in-time {\it mild} solutions of \eqref{eq}--\eqref{eq-ini} which are solutions of  the following integral equation
\begin{align}
u(t) = e^{t\Delta} u_0 - \int_0^t \nabla e^{(t-s)\Delta} \big( u\nabla v\big)(s)\, \dy{s} \label{duhamel}
\end{align}
where $e^{t\Delta}$ is the Neumann heat semigroup in $\Omega$. Moreover, we use the following estimates of $\{e^{t\Delta}\}_{t\geq 0}$.

\begin{lemm}
Let  $\lambda_1>0$ denote the first nonzero eigenvalue of $-\Delta$ in $\Omega$ under Neumann boundary conditions. Then there exist constants $C_1$, $C_2$ independent of $t,f$ which have the following properties.
\begin{itemize}
\item [(i)] If  $1 \le q \le p \le +\infty$ then
\begin{align}\label{G1}
\norm{{e^{t\Delta} f}}{p} \le C\big{(}1+ t^{- \frac{d}{2}\left( \frac{1}{q}-\frac{1}{p} \right)}\big{)}\norm{f}{q}
\end{align}
holds for all $f \in L^q (\Omega)$.
\item [(ii)] If  $1 \le q \le p \le +\infty$ then
\begin{align}\label{G2}
\norm{{\partial_x e^{t\Delta} f}}{p} \le Ct^{- \frac{d}{2}
 \left(\frac{1}{q}-\frac{1}{p} \right)- \frac{1}{2}}e^{-\lambda_1 t}\norm{f}{q} 
\end{align}
is true for all $f \in L^q (\Omega)$.
\end{itemize}
\end{lemm}
Proofs of above inequalities \eqref{G1} and \eqref{G2} are well-known and can be found {\it e.g.} in \cite{rothe}.

First, we construct global-in-time solutions in the case of mildly singular kernel.

\begin{proof}[Proof of Theorem \ref{local-ex1}]
We split the proof into two parts. First we construct the local-in-time solution to problem \eqref{eq}--\eqref{eq-ini} and later on we show how to extend this solution on every time interval $[0,T]$.

{\it Step 1. Local-in-time solution.}
Here, we follow the reasoning from \cite[Theorem 2.2]{KS11}. We construct the local-in-time solution to the equation \eqref{duhamel}, written as 
$u(t) =e^{t\Delta} u_0 + B(u, u)(t)$ with the bilinear form
\begin{align}
B(u, v) (t) = - \int_0^t \nabla e^{(t-s)\Delta} \big( u\nabla \K(v)\big)(s)\, \dy{s}, \label{sec-ex-eq1}
\end{align}
in the space $\mathcal{Y}_T$. Notice that  $e^{t\Delta} u_0 \in \mathcal{Y}_T$ by \eqref{G1}. To apply ideas from \cite[Theorem 2.2]{KS11}, one should prove the following estimates of the bilinear form \eqref{sec-ex-eq1}.

First, let us notice that by Minkowski's inequality we have that
\begin{align}\label{G_q_1}
\|\nabla\K(v)\|_{q^\prime}\leq\|\int_\Omega |\nabla_x K(\cdot,y)|v(y)\dy{y}\|_{q^\prime}\leq \int_\Omega \|\nabla_x K(\cdot,y)\|_{q^\prime}|v(y)|\dy{y}\leq \|\nabla_xK\|_{\infty,q^\prime}\|v\|_1
\end{align}
Now, for every $u, v \in \mathcal{Y}_T$, using \eqref{G2} combined with relation \eqref{G_q_1} we obtain 
\begin{align*}
\|B(u, v) (t)\|_1 %
&\le C\int_0^t (t-s)^{-1/2} \|u\nabla \K(v)(s)\|_1 \dy{s} \\
&\le C\int_0^t (t-s)^{-1/2} \|u (s)\|_q \|\nabla\K(v)(s)\|_{q^\prime} \dy{s}\\
&\le C\|\nabla_xK\|_{\infty,q^\prime}\int_0^t (t-s)^{-1/2} \|u (s)\|_q\|v (s)\|_1 \dy{s}.
\end{align*}
Therefore, by the argument using in \cite{KS11} we obtain 
\begin{align}\label{ex-201}
\|B(u, v) (t)\|_1  \le C T^{\frac{1}{2}(1-d(1-\frac{1}{q}))} \|\nabla_xK\|_{\infty,q^\prime}\|u\|_{\mathcal{Y}_T} \|v\|_{\mathcal{Y}_T}
\end{align}
where $\frac{1}{2}(1-d(1-\frac{1}{q}))>0$.

In a similar way, we prove the following $L^q$-estimate
\begin{align}
t^{\frac{d}{2}(1-\frac{1}{q})}\|B(u, v) (t)\|_q 
&\le Ct^{\frac{d}{2}(1-\frac{1}{q})} \int_0^t(t-s)^{-1/2}\|u\nabla\K(v)(s)\|_q\dy{s} \label{ex-101}\\
&\le
Ct^{\frac{d}{2}(1-\frac{1}{q})} \int_0^t(t-s)^{-1/2}\|u\|_q\|\nabla\K(v)(s)\|_\infty\dy{s} \nonumber\\
&\le
C\|\nabla_xK\|_{\infty,q^\prime}t^{\frac{d}{2}(1-\frac{1}{q})} \int_0^t(t-s)^{-1/2}\|u(s)\|_q\|v(s)\|_q\dy{s} \nonumber
\end{align}
since
\begin{align*} 
\|\int_\Omega \nabla_x K(\cdot,y)v(y)\dy{y}\|_\infty\leq {\rm ess}\sup_{x\in\Omega} \|\nabla_x K(x,\cdot)\|_{q^\prime}\|v\|_q.
\end{align*}
Again, by the argument using in \cite{KS11} we obtain
\begin{align}\label{ex-202}
t^{\frac{d}{2}(1-\frac{1}{q})}\|B(u, v) (t)\|_q  \le CT^{\frac{1}{2}(1-d(1-\frac{1}{q}))}\|\nabla_xK\|_{\infty,q^\prime}\|u\|_{\mathcal{Y}_T} \|v\|_{\mathcal{Y}_T}.
\end{align}
By inequalities \eqref{ex-201} and \eqref{ex-202} we obtain the following estimate of the bilinear form
\begin{align*}
 \|B(u, v) \|_{\mathcal{Y}_T} \le CT^{\frac{1}{2}(1-d(1-\frac{1}{q}))} \|\nabla_xK\|_{\infty,q^\prime} \|u\|_{\mathcal{Y}_T} \|v\|_{\mathcal{Y}_T}.
\end{align*}
Hence, choosing $T > 0$ such that 
$
4 CT^{\frac{1}{2}(1-d(1-\frac{1}{q}))}\|\nabla_xK\|_{\infty,q^\prime}\|u_0 \|_1 < 1,
$
we obtain the  solution in $\mathcal{Y}_T$ by \cite[Lemma 3.1]{KS11}.

{\it Step 2. Global solution.}
Now, it suffices to follow a standard procedure which consists in applying repeatedly previous step to equation \eqref{eq} supplemented with the initial datum $u(x,kT)$ to obtain a unique solution on the interval $[kT,(k+1)T]$ for every $k\in\mathbb{N}$. Notice, that we can pass this procedure since the local existence time $T$ depends only on $\|u_0\|_1$ and $\|\nabla_xK\|_{\infty,q^\prime}$ which implies that it does not change for all nonnegative $u_0\in L^1(\Omega)$ with the same $L^1$-norm (see Remark \ref{th-mass}).

\end{proof}

Now, we prove local-in-time existence of solutions in the case that $\K$ is strongly singular.

\begin{proof}[Proof of Theorem \ref{local-ex2}]
We assume now, that $q^\prime\in[1,d]$. Again notice that  $e^{t\Delta} u_0 \in \mathcal{X}_T$ since by \eqref{G1} we have
\begin{align*}
\|e^{t\Delta} u_0\|_{\mathcal{X}_T}\leq C(\|u_0\|_1+\|u_0\|_q).
\end{align*}

Next, for every $u, v \in \mathcal{Y}_T$, we get 
\begin{align*}
\|B(u, v) (t)\|_1 %
&\le C\int_0^t (t-s)^{-1/2} \|u\nabla \K(v)(s)\|_1 \dy{s} \\
&\le C\int_0^t (t-s)^{-1/2} \|u (s)\|_q \|\nabla\K(v)(s)\|_{q^\prime} \dy{s}\\
&\le C\|\nabla_xK\|_{\infty,q^\prime}\int_0^t (t-s)^{-1/2} \|u (s)\|_q\|v (s)\|_1 \dy{s}\\
&\le CT^{1/2}\|\nabla_xK\|_{\infty,q^\prime} \|u\|_{\mathcal{X}_T}\|v\|_{\mathcal{X}_T}
\end{align*}
where $C$ is a positive constant.

To deal with the $L^q$-norm of $B(u,v)$ we proceed similarly
\begin{align*}
\|B(u, v) (t)\|_q %
&\le C\int_0^t (t-s)^{-1/2} \|u\nabla \K(v)(s)\|_q \dy{s} \\
&\le C\|\nabla_xK\|_{\infty,q^\prime}\int_0^t (t-s)^{-1/2} \|u (s)\|_q\|v (s)\|_q \dy{s}\\
&\le CT^{1/2}\|\nabla_xK\|_{\infty,q^\prime} \|u\|_{\mathcal{X}_T}\|v\|_{\mathcal{X}_T}.
\end{align*}
Summing up these inequalities, we obtain the following estimate of the bilinear form
\begin{align*}
 \|B(u, v) \|_{\mathcal{X}_T} \le CT^{1/2}\|\nabla_xK\|_{\infty,q^\prime}  \|u\|_{\mathcal{X}_T} \|v\|_{\mathcal{X}_T}.
\end{align*}
Hence, choosing $T > 0$ such that 
$
4CT^{1/2}\|\nabla_xK\|_{\infty,q^\prime}(\|u_0 \|_1+\|u_0\|_q) < 1,
$
we obtain the  solution in $\mathcal{X}_T$ by \cite[Lemma 3.1]{KS11}.
\end{proof}

\end{document}